\newif\ifsmfart
\IfFileExists{smfart.cls}
  {\documentclass[12pt,english]{smfart}
   \IfFileExists{smfenum.sty}{\usepackage{smfenum}}{}
    \usepackage{bull}
    \smfarttrue}
{\message{^^J*** It would be better to typeset this file with smfart.cls ***^^J^^J}

\documentclass[12pt]{amsart}}
\usepackage{amssymb, amsfonts, amsthm, amscd}
\usepackage{epsfig}
\numberwithin{equation}{section}

\usepackage{epsfig}
\usepackage{graphics}
\input xy
\xyoption{all}

\def\Title    {Unramified cohomology of finite groups of Lie type}
\def\Author   {Fedor Bogomolov, Tihomir Petrov and Yuri Tschinkel}
\def\Subject  {Algebraic geometry}
\def\Keywords {Rationality}

\usepackage{ifpdf}
      \ifpdf
\usepackage[pdftex,
hyperindex=true,
pdftitle={\Title},pdfauthor={\Author},pdfsubject={\Subject},
pdfkeywords={\Keywords},pdfpagemode={UseOutlines},
bookmarks=true,bookmarksopen=true,
bookmarksopenlevel=0,bookmarksnumbered=true,
pdfstartview={FitV},
colorlinks=true
]{hyperref}
 \else
 \usepackage{hyperref}
 \fi

\usepackage{graphicx}

\advance\headheight 2pt
\calclayout

\theoremstyle{plain}
\newtheorem{prop}[subsection]{Proposition}

\newtheorem{thm}[subsection]{Theorem}
\newtheorem{coro}[subsection]{Corollary}

\newtheorem{lemm}[subsection]{Lemma}

\theoremstyle{definition}

\newtheorem{conj}[subsection]{Conjecture}

\theoremstyle{remark}
\newtheorem{rem}[subsection]{Remark}

\newtheorem{exam}[subsection]{Example}

\newcommand{\C}{\Bbb C}

\newcommand{\Z}{\Bbb Z}
\newcommand{\ra}{\rightarrow}
\def\KK{\boldsymbol{K}}
\def\I{{\mathcal I}}

\def\ba{\backslash}

\newcommand{\Aut}{\operatorname{Aut}}
\newcommand{\G}{\operatorname{G_2}}

\newcommand{\GL}{GL}

\newcommand{\Ker}{\operatorname{Ker}}

\newcommand{\Syl}{\mathrm{Syl}}
\newcommand{\Val}{\operatorname{Val}}

\def\cV{{\mathcal V}}
\def\G{{\mathcal G}}

\def\char{{\rm char}}
\def\deg{{\rm deg}}

\def\A{{\mathbb A}}
\def\C{{\mathbb C}}
\def\F{{\mathbb F}}

\def\Z{{\mathbb Z}}
\def\C{{\mathbb C}}
\def\N{{\mathbb N}}

\begin{document}
\title[Unramified cohomology of finite groups of Lie type]{Unramified cohomology of finite 
groups of Lie type}
\author{Fedor Bogomolov}
\address{Courant Institute of Mathematical Sciences\\
251 Mercer Street\\
New York, NY 10012--1185, USA}
\email{bogomolo@cims.nyu.edu}

\author{Tihomir Petrov}
\address{Department of Mathematics\\
University of California, Irvine\\
Irvine, CA 92697-3875, USA}
\email{tpetrov@math.uci.edu}

\author{Yuri Tschinkel}
\address{Courant Institute of Mathematical Sciences\\
251 Mercer Street\\
New York, NY 10012--1185, USA}
\email{tschinkel@cims.nyu.edu}



\date{\today}


\begin{abstract}
We prove vanishing results for unramified stable cohomology of finite groups of Lie type.
\end{abstract}

\maketitle

\tableofcontents

\section{Introduction}
\label{sect:intro}

Let $k$ be an algebraically closed field,  
$G$ a finite group and $V$ a faithful representation of $G$ over $k$. 
In this note we compute cohomological obstructions to 
stable rationality of quotients of $V$ by $G$ 
introduced by Saltman \cite{saltman} and 
\cite{bogomolov} and studied in  
\cite{ct}, \cite{peyre},
\cite{bmp}. 

\

Let $K=k(V)^G$ be the function field
of the quotient variety and 
$$
s\,:\,  \G_K \to G
$$
the natural homomorphism 
from the absolute Galois group of $K$ to $G$. 
We have an induced map on 
cohomology with coefficients 
in the torsion group $\Z/\ell$, with trivial $G$-action, 
$$
s^*_i:  H^i(G,\Z/\ell)\to H^i(\G_K,\Z/\ell). 
$$
Note that $s_i^*$ depends on the ground field $k$, but not
on the choice of the faithful representation $V$ over that field.  
The groups 
$$
H^i_{k,s} (G,\Z/\ell):=H^i(G,\Z/\ell)/ \Ker(s^*_i),
$$
are called {\em stable cohomology groups} over $k$. 
They form a finite ring.
We may consider them as subgroups of $H^i(\G_K,\Z/\ell)$.  
Every divisorial valuation $\nu\in \Val_K$ of $K$ defines a residue map 
$$
\partial_{\nu}: H^i(\G_K, \Z/\ell)\to H^{i-1}(\G_{\KK_{\nu}}, \Z/\ell),
$$
where $\KK_{\nu}$ is the residue field of $\nu$. 
The groups 
$$
H^i_{k,un} (G,\Z/\ell):=\bigcap_{\nu\in \Val_K}\Ker(\partial_{\nu}\circ s_i^*) 
\subset H^i_{k,s} (G,\Z/\ell)
$$
form  a subring of $H^*_{k,s} (G,\Z/\ell)$. 
A basic fact is that if there exists a faithful representation $V$ of $G$ 
over $k$ and a unirational parametrization of the quotient $V/G$ 
whose degree is prime to $\ell$,
then 
$$
H^i_{k,un}(G,\Z/\ell)=0,\,\, \text{ for  all  } i>0.
$$ 
In particular, these cohomology groups vanish if this 
quotient is stably rational. 

\

For example, the rings of invariants of finite groups generated by 
pseudo-reflections are polynomial, 
the corresponding quotient varieties rational, and the cohomological invariants
trivial. In particular, all Weyl groups $\mathcal W$ of semi-simple
Lie groups have 
$$
H^i_{k,un}(\mathcal W,\Z/\ell)=0,\,\, \text{ for  all  } i>0, \,\,
\text{ and all } \,\, k.
$$ 

\begin{conj}
\label{conj:main}
Let $G$ be a finite simple group. Then 
$$
H^i_{k,un}(G,\Z/\ell)=0,\,\, \text{ for all }\,\, i>0, 
\,\,\text{ all }\,\, k
\,\, \text{ and  all }\,\, 
\ell.
$$ 
\end{conj}

The $i=2$ case of this conjecture was proved for  
$G=\mathsf{PSL}_n(\F_q)$ and $k=\mathbb C$ 
in \cite{bmp} and for simple and  quasi-simple groups of Lie type in \cite{kun}.
Examples of functions fields with vanishing second
and nonvanishing third unramified cohomology were given in \cite{peyre-un}.
Here we prove that many of these cohomology groups vanish
for finite groups of Lie type, for $k=\bar{\F}_q$. In fact, we
prove the stable rationality of many associated quotients spaces.
Our main theorem is:

\begin{thm}
\label{thm:stab-rat}
Let $G$ be one of the following groups
$$
\mathsf{SL}_n(\F_q), \quad \mathsf{Sp}_{2n}(\F_q), 
$$
or their twisted forms
$$
{}^2\mathsf{SL}_n(\F_q), \quad {}^2\mathsf{Sp}_{2n}(\F_q).
$$
Let $V$ be a faithful representation of $G$ over $k=\bar{\F}_p$. 
Then the quotient of $V$ by $G$ is stably rational over $k$. 
\end{thm}

\noindent
In particular, Conjecture~\ref{conj:main} holds in these cases
for $\ell\nmid q$ and $k=\bar{\F}_p$. 
Our main tool is a theorem of Lang which 
proves the rationality of certain quotient spaces over $\bar{\F}_p$.

\begin{thm}
\label{thm:2}
Let $\mathsf G$ be a semi-simple simply-connected 
Lie group defined over a finite field $\F_q$. 
Then the image
$$
H^i(\mathsf G(\F_q),\Z/\ell)\ra H_{k,s}(\mathsf G/\mathsf G(\F_q),\Z/\ell)
$$
is zero, for $k=\bar{\F}_q$, all $i>0$ and $\ell\nmid q$. 
\end{thm}

Combining this with results of Tits \cite{tits} we obtain the following:

\begin{thm}
\label{thm:main-intr}
Let $G$ be a finite quasi-simple group of 
Lie type over a finite field of characteristic $p$. 
Put
$$
d(G):=\left\{ \begin{array}{ll} \{p\} &  \text{ if } G \text{ is of type } 
A_n,B_n,C_n, D_n \text { or } G_2;\\
\{2,3,p\}    &  \text{ if } G \text{ is of type } F_4,E_6,E_7;\\
\{2,3,5, p\} & \text{ if } G \text{ is of type } E_8. 
\end{array}
\right.
$$
Then for all algebraically closed fields $k$ one has
$$
H^i_{k,un}(G,\Z/\ell)=0,\,\, 
\text{ for all }\,\, i>0, \,\, \text{ and all }\,\, \ell\notin d(G). \,\,
$$  
For $k=\bar{\F}_p$ the vanishing also holds for $\ell = p$. 
\end{thm}

\

Here is the roadmap of the paper. In Section~\ref{sect:gen} we study the birational
type of quotients $G\ba\mathsf{G}/H$, where $\mathsf G$ is an algebraic group
over an algebraically closed field $k$ and $G,H\subset \mathsf{G}(k)$
are finite subgroups, acting on $\mathsf{G}$ by 
translations on the left, resp. on the right.
In Section~\ref{sect:class} we study the classical groups. 
In Section~\ref{sect:coho} we introduce stable and unramified cohomology over 
arbitrary algebraically closed fields and prove their basic properties. 
In Section~\ref{sect:gen-van} we establish general vanishing results, 
applying theorems of Lang and Tits. 
In Section~\ref{sect:splitting} we sketch another approach which is using
the structure of Sylow subgroups of quasi-simple groups
of Lie type. As an example we prove the triviality
of unramified cohomology over $\C$ for all groups $\GL_n(\mathbb{F}_q)$
and $\ell$ coprime to $q$.

\

\noindent
{\bf Acknowledgments:} The first author was supported by NSF grant
DMS-0701578 and the third author by NSF grants DMS-0554280 and DMS-0602333.


\section{Equivariant birational geometry}
\label{sect:gen}

We work over an algebraically closed field $k$. 
We say that $k$-varieties $X$ and $Y$ are stably birational, and write 
$X\sim Y$, if $X\times \A^n$ is birational to $Y\times \A^m$, 
for some $n,m\in \mathbb N$.

\

Let $G$ be an algebraic group and $X$ an algebraic variety over $k$, 
with a $G$-action
$$
\lambda\,:\, G\times X\ra X.
$$ 
We will sometimes consider different actions of the same group. 
To emphasize the action we will write $\lambda(G)\ba X$ for 
the quotient of $X$ by the $\lambda$-action of $G$;
we write  $G\ba X$, when the action is clear from the context.

\

We say that the action of $G$
is {\em almost free} if 
there exists a Zariski open subset $X^{\circ}\subset X$
on which the action is free. In particular, the quotient map 
$X\ra G\ba X$ is separable.

\begin{exam}
\label{exam:basic}
Let $V$ be a faithful complex representation of $G$. 
Then $G$ acts almost freely on $V$. 
\end{exam}

\begin{lemm}
\label{lemm:openU}
Let $G$ be a finite group and $V$ a faithful representation of $G$ over an 
algebraically closed field $k$. 
Let $Y$ be an affine variety over $k$, with a free $G$-action, and 
$y\in Y(k)$ a point. 

For every Zariski open  $U\subset V$ there exist 
a $G$-equivariant $k$-morphism $\phi_U\,:\, Y\ra V$
and a Zariski open $G$-invariant subset $Y^{\circ}\subset Y$ such that 
\begin{itemize}
\item $y\in Y^{\circ}(k)$;
\item $\phi_U(Y^{\circ})\subset U$.
\end{itemize}
\end{lemm}

\begin{proof}
It suffices to consider $V:=\A_k^{|G|}$, the affine $k$-space, with the 
induced faithful $G$-action. For any divisor $D\subset V$
there exists a Zariski open subset $U\subset V$ such that for every 
point $v\in U(k)$ its $G$-orbit $G\cdot v\notin D$. 
For every Zariski open  $U\subset V$ 
there exists a $G$-equivariant $k$-morphism $\phi_U\,:\, Y\ra V$ such that $\phi_U(y)\in U(k)$
(functions separate points). 
This implies the existence of a Zariski open $G$-invariant subset $Y^{\circ}\subset Y$
with the claimed properties.
 \end{proof}

A $G$-variety $X$ is called {\em $G$-affine}, and the corresponding action {\em affine}, 
if there exists a $G$-equivariant birational isomorphism between $X$ and a faithful 
representation of $G$.
Let $\cV$ and $X$ be affine $G$-varieties. A $G$-morphism 
$\pi\,:\, \mathcal V\ra X$ is called an {\em affine $G$-bundle} if
it is an affine bundle over some open subset $X^{\circ}\subset X$
and the $G$-action is compatible with this structure of an affine bundle.

By Hilbert 90, an affine $G$-bundle $\cV\ra X$ is $G$-birational to a
finite dimensional $G$-representation over the function field of 
$K=k(X)$, compatible with the given $G$-action on $K$. 
A morphism $\rho\,:\, X\ra B$ of $G$-varieties will be called a 
$G$-ruling (and $X$ - $G$-ruled) over $B$ if there exists a finite set of 
affine $G$-varieties 
$$
X_n=X,B_{n-1},X_{n-1},B_{n-2},X_{n-2},\ldots, X_1,B_0=B
$$ 
such that 
$X_i\ra B_{i-1}$ is an affine $G$-bundle
and $B_i\subset X_i$ a $G$-stable Zariski open subset, for $i=1,\ldots, n$.

\begin{lemm}
\label{lemm:g-affine}
Assume that $\rho \,:\, X\ra B$ is a $G$-ruling over $B$ and 
that the action of $G$ on $X$ is almost free. Then 
$X$ is $G$-affine. 
\end{lemm}

\begin{proof}
Follows from Hilbert 90. 
 \end{proof}

Let $X,Y$ be smooth varieties with an almost free action of $G$. 
We write $X\stackrel{G}{\leadsto} Y$ if there exist a
$G$-representation $V$, a Zariski open $G$-stable
subset $X^{\circ}\subset X$ and a $G$-morphism (not necessarily dominant)
$\beta\,:\, X^{\circ}\times V\ra Y$.  We write $X\stackrel{G}{\leftrightsquigarrow} Y$, 
and say that the $G$-actions are equivalent, 
if $X\stackrel{G}{\leadsto} Y$ and $Y\stackrel{G}{\leadsto} X$. 

\begin{lemm}
\label{lemm:tran}
If $X\stackrel{G}{\leadsto} Y$, then the morphism 
$$
\beta\,:\, G\ba (X\times Y)\ra G\ba X
$$
has a rational section. 
\end{lemm}

\begin{proof}
Consider the morphism 
$$
\beta'\,:\, G\backslash (X\times V \times Y)\ra G\backslash X,
$$
where $G$ acts diagonally. The graph of the map 
$X\ra (Y\times V)$ is $G$-stable and gives
a section of $\beta'$. The projection of this section 
to $G\ba (X\times Y)$ is a section of $\beta$. 
 \end{proof}


\

\begin{lemm}
\label{lemm:lie}
Let $\mathsf{G}$ be a Lie group over an algebraically closed field 
$k$. Let $G\subset \mathsf G(k)$ be a finite subgroup. 
Let $X$ be an algebraic variety over $k$ with an almost free
action of $G$. Assume that $X\stackrel{G}{\leadsto} \mathsf{G}$, 
where $\mathsf{G}$ is considered as a $G$-variety, with 
a left action. 
Then 
$$
G\ba (X\times \mathsf{G}) \sim  G\ba X.
$$ 
\end{lemm}

\begin{proof}
By Lemma~\ref{lemm:tran}, there is a Zariski open 
$G$-stable subset $X^{\circ}\subset X$ so that the $G$-morphism
(projection to the first factor)
$$
\beta\,:\, G\ba (X\times \mathsf{G}) \ra (G \ba X) 
$$
has a section. 
We also have a {\em right} action of $\mathsf{G}$, 
which preserves the fibration structure given by $\beta$. 
Thus it is a principal homogeneous space over $G\ba X^{\circ}$, 
for some $G$-stable Zariski open $X^{\circ}\subset G$, with a section.
Hence it is birational to $(G\ba X)\times \mathsf{G}$. 
It suffices to recall that $\mathsf{G}$ is rational over $k$.
 \end{proof}

Let $\mathsf{G}$ be a connected algebraic group and  
$F\in \Aut_k(\mathsf{G})$ a $k$-automorphism of $\mathsf{G}$. 
Let $G\subset \mathsf{G}(k)$ be a finite subgroup, with a natural
left action 
$$
\begin{array}{rccc}
\lambda \,:   & G\times \mathsf{G} & \ra &  \mathsf{G} \\
               & (\gamma,g) & \mapsto & \gamma \cdot g 
\end{array} 
$$
We also have an $F$-twisted right action 
$$
\begin{array}{rccc}
\rho^F \,:  & G\times \mathsf{G} & \ra     &  \mathsf{G} \\
                & (\gamma,g)         & \mapsto &  g \cdot F(\gamma^{-1}) 
\end{array} 
$$
and an $F$-conjugation
$$
\begin{array}{rccc}
\kappa^F \,:  & G\times \mathsf{G} & \ra     &  \mathsf{G} \\
                & (\gamma,g)         & \mapsto & \gamma \cdot g \cdot F(\gamma^{-1}). 
\end{array} 
$$

\begin{lemm}
\label{lemm:twist}
Assume that $G\subset \mathsf{G}(k)$ has the following properties:
\begin{itemize}
\item[(1)] there exists a faithful $G$-representation $V$
such that $V\stackrel{G}{\leadsto} \mathsf{G}$, 
where $G$ acts on $\mathsf{G}$ via $\lambda$;
\item[(2)] the twisted action $\rho^F$ on $\mathsf{G}$ is almost free.
\end{itemize}  
Then the quotient of $\mathsf G$ by the $F$-twisted conjugation $\kappa^F$ 
of $G$ is stably birational to the quotient of $\mathsf G$ by $\lambda$.  
\end{lemm}

\begin{proof}
Consider the diagonal action of $G$ on $\mathsf{G}\times \mathsf{G}$:
$$
\begin{array}{ccc}
G\times \mathsf{G} \times \mathsf{G}& \stackrel{(\lambda, \rho^F)}{\longrightarrow} &  
\mathsf{G}\times \mathsf{G} \\
(\gamma,g,g')                       & \mapsto &  (\gamma \cdot g, g'\cdot F(\gamma^{-1})). 
\end{array} 
$$
Let $\Delta^F:=\{(g,F^{-1}(g))\} \subset \mathsf{G}\times \mathsf{G}$ 
be the $F$-twisted (anti)diagonal. Then $\Delta^F$ is preserved under 
the $(\lambda,\rho^F)$-action of $G$ and descends to a section of 
the principal (right) $\mathsf{G}$-fibration 
$$
(\lambda, \rho^F)(G)\ba (\mathsf{G}\times \mathsf{G}) \ra  \rho^F(G)\ba \mathsf{G},
$$
projection to the second factor.
It follows that 
$$
\rho^F(G)\ba \mathsf{G} \sim (\lambda, \rho^F)(G)\ba (\mathsf{G}\times \mathsf{G}).
$$
Observe that 
$$
\rho^F(G)\ba \mathsf{G} \sim \lambda(G)\ba \mathsf{G}.
$$

Now we show that
\begin{equation}
\label{eqn:kappa}
\kappa^F(G)\ba \mathsf{G} \sim (\lambda, \rho^F)(G)\ba (\mathsf{G}\times \mathsf{G}).
\end{equation}
Let $V$ be a faithful representation of $G$ as in (1) and $V^{\circ}\subset V$
a $G$-stable Zariski subset admitting a $G$-map into $\mathsf{G}$, considered
with the $\lambda$-action of $G$.  
We know that there exists a $G$-morphism $\xi\,:\, \mathsf{G}\ra V$, 
where $\mathsf{G}$ is considered with the $\kappa^F$-action of $G$, 
such that $\xi(\mathsf{G})\cap V^{\circ}\neq \emptyset$ (see Lemma~\ref{lemm:openU}). 
It follows that $\mathsf{G}\stackrel{G}{\leadsto} \mathsf{G}$, 
where the source carries the $\kappa^F$-action of $G$ and the image the 
$\lambda$-action of $G$. 
Equation~\eqref{eqn:kappa} now follows from Lemma~\ref{lemm:tran}.
 \end{proof}

\begin{coro}
\label{coro:twist}
Let $G\subset \mathsf{G}(k)$  by a finite subgroup
satisfying Assumption (1) of Lemma~\ref{lemm:twist}.
Let $V$ be a faithful representation of $G$ over $k$.
Then 
$$
G\ba \mathsf{G} \sim  G\ba V.
$$
More generally, for any $X$ with an almost free action of $G$
we have
$$
G\ba (X\times V)\sim G\ba X.
$$ 
\end{coro}

\begin{proof}
Note that $G\ba (\mathsf{G}\times V)$ is
a vector bundle over $G\ba \mathsf{G}$, and
hence stably birational to it. 
On the other hand, it is
a right  $\mathsf{G}$-fibration over $G\ba V$ with section defined
by the $G$-equivariant map $V\to \mathsf{G}$.

To prove the statement for $X$ it suffices to notice that a dense
Zariski open $G$-stable subset $X^{\circ}\subset X$ admits a nontrivial 
$G$-morphism to $V$. 
 \end{proof}

\section{Equivariant birational geometry of classical groups}
\label{sect:class}

In this section, $k$ is an algebraically closed field, of any characteristic.

\begin{lemm}
\label{lemm:sl2-stable}
The conjugation action $\kappa\,:\, \mathsf{SL}_2\ra \mathsf{SL}_2$
is equivalent to a linear action. 
\end{lemm}

\begin{proof}
Realize $\mathsf{SL}_2$ as a nonsingular quadric in $\A^4 = \mathsf{M}_2$.
The conjugation action is linear on $\mathsf{M}_2$ and has
a fixed point corresponding to the identity.
The projection of $\mathsf{SL}_2$ from the identity to the locus 
of trace zero matrices is equivariant
and has degree 1. Hence the conjugation action is rationally equivalent
to the action on trace zero matrices. 
 \end{proof}

\begin{lemm}
\label{lemm:z2-rat}
Let $G=(\Z/2)^n$ and $X$ be a $G$-affine variety over $k$.
Then $G\ba X$ is rational.
\end{lemm}


\begin{lemm}
\label{lemm:pgl2}
Let $G\subset \mathsf{PGL}_2(k)\simeq 
\mathsf{SO}_3(k)\subset \mathsf{M}_{2}(k)$ be a finite subgroup. 
Then the action of $G$ on $\mathsf{PGL}_2$ by conjugation is equivalent to the 
left (linear) action of $G$ on trace zero matrices 
$\mathsf{M}_2^0\subset \mathsf{M}_2$. In particular, 
$G\ba \mathsf{PGL}_2$ is stably rational.  
\end{lemm}

\begin{lemm}
\label{lemm:s03}
The left action of $(\Z/2)^2\subset \mathsf{SO}_3\hookrightarrow \mathsf{SO}_4$ is
linear.
\end{lemm}

\begin{proof}
Consider the subgroup $(\Z/2)^3 \subset \mathsf{SO}_4$.
It contains a central subgroup $\Z/2$ and a complementary
subgroup $H_2=(\Z/2)^2\subset \mathsf{SO}_3 = \mathsf{PGL}_2$. 
By Lemma~\ref{lemm:sl2-stable}, the conjugation action 
$H_2\subset \mathsf{PGL}_2$ is linear.

This is a subgroup of the diagonal subgroup
$\mathsf{SO}_3 = \mathsf{PGL}_2\subset \mathsf{SL}_2\times \mathsf{SL}_2/(\Z/2)$.
Note that $\mathsf{SO}_4$ is a product of $\mathsf{SO}_3= (h,h)$ and
$\mathsf{Spin}_3 = (g,1)$, and that conjugation by elements in $(\Z/2)^2$
respects this decomposition. 
Thus the action is a vector bundle over 
$\mathsf{SL}_2/(\mathsf Q_8)_{conj} = \mathsf{SL}_2/(\Z/2)^2$
(where $\mathsf Q_8$ are the quaternions).

Hence the action on $\mathsf{SO}_4$ is linear and the automorphism $F$
is the identity on the diagonal $\mathsf{SO}_3$. The same holds for
the twisted $F$-action.
 \end{proof}

\begin{lemm}
\label{lemm:itself}
Let $\mathsf{G}$ be an algebraic group over $k$. Assume that 
$\mathsf{G}$ admits an affine action on itself, e.g., 
$\mathsf{G}=\mathsf{GL}_n, \mathsf{SL}_n, \mathsf{Sp}_n$. 
Let $G\subset \mathsf{G}(k)$ be a finite subgroup which has trivial
intersection with the center of $\mathsf{G}$. 

Then the conjugation action of $G$ on $\mathsf{G}$
is stably birationally equivalent to a linear action.
\end{lemm}

\begin{proof}
By Corollary~\ref{coro:twist}, 
the diagonal left translation action of 
$G$ on $\mathsf{G}\times \mathsf{G}$ is equivalent
to the action on a principal
$\mathsf{G}$-bundle over $\mathsf{G}$, with 
$G$ acting on the base by conjugation. This proves the equivalence.
 \end{proof}

\begin{coro}
\label{coro:itself}
Let $G\subset \mathsf{G}(k)$ 
be a finite subgroup
as in Lemma~\ref{lemm:itself}. 
Assume that the action of $G$ on $\mathsf{G}$ by left translations
is affine. Then both the left and the conjugation action of $G$ 
on $\mathsf{G}$ are stably birational to a linear action.
\end{coro}

\begin{prop}
\label{prop:g-afff}
Let $\mathsf G$ be a classical simply-connected Lie group of type $A$ or $C$, i.e., 
$\mathsf G=\mathsf{SL}_n$ or $\mathsf G=\mathsf{Sp}_{2n}$ over $k$. 
Let $G\subset \mathsf G(k)$ be a finite subgroup. 
Then $\mathsf G$ is a $G$-affine variety, for the standard left action of $G$. 
\end{prop}

\begin{proof}
Note first that any finite subgroup $G\subset \mathsf{GL}_n(k)$ induces
a $G$-affine structure on $\mathsf{GL}_n$. Indeed, 
$\mathsf{GL}_n\subset \mathsf{M}_{n\times n}=\oplus_{i=1}^n V^{(i)}$, 
a direct sum of $n$ copies of the standard representation 
$V$ of $\mathsf{GL}_n$. 
Put $B_0=0$, a point, and $X_1:=V^{(1)}=V$. We have a canonical projection $X_1\ra B_0$. 
Define $X_j\subset \oplus_{i=1}^j V^{(i)}$ as the set of those vectors, whose projections to
$\oplus_{i=1}^{j-1}V^{(i)}$ are linearly independent, and $B_j\subset X_j$ as the subset of 
vectors in $X_j$, which are linearly independent in $\oplus_{i=1}^j V^{(i)}$, under the
standards identification $V=V^{(i)}$. For any $G\subset \mathsf{GL}_n$ this defines the structure
of a $G$-ruling on $\mathsf{GL}_n$ over a point. 

For $\mathsf G=\mathsf{SL}_n$, and $G\subset \mathsf G$ we have a similar $G$-ruling: 
For $j=1,\ldots, n-1$ it is the same as above. For $j=n$, put $X_n:=\mathsf{SL}_n$.
The map $X_n\ra B_{n-1}$ is the restriction of the map above. Explicitly, 
it is the projection to the first $(n-1)$-vectors $(v_1,\ldots, v_{n-1})$, 
with fiber an affine subspace $F_{(v_1,\ldots, v_{n-1})}\subset V^{(n)}=V$, 
given by the affine equation in the coordinates of the last vector $\det(v_1,\ldots, v_n)=1$.
Now we can apply Lemma~\ref{lemm:g-affine} to conclude that $\mathsf{SL}_n$ is $G$-affine.   

Note that by Tsen's theorem, the morphism $\mathsf{GL}_n\ra \mathsf{GL}_n/\mathsf{SL}_n$ 
has a section. This gives a $G$-equivariant birational isomorphism 
$\mathsf{SL}_n\times \mathbb G_m\ra \mathsf{GL}_n$.

\

The group $\mathsf{G}=\mathsf{Sp}_{2n}$ has a canonical embedding into $\mathsf{M}_{2n\times 2n}$, 
defined by the equations
\begin{equation}
\label{eqn:system}
\omega(v_i,v_{i'})=\delta_{i',n+i}, \,\, \text{ for }\,\, i<i',
\end{equation}
($\omega$ is the standard bilinear form and $\delta$ is the delta function). 
The system of projections is induced from the one above:
$$
X_j=\{(v_1, \ldots, v_j)\} \subset \oplus_{i=1}^j V^{(i)},
$$ 
satisfying equations \eqref{eqn:system}, 
for indices $1\le i<i'\le j$, and the property that the vectors
$v_1,\ldots, v_{j-1}$ are linearly independent in $V$, under the identifications $V^{(i)}=V$.  
The subvariety $B_j\subset X_j$ is given as the locus where $v_1,\ldots, v_j$ are
linearly independent. Each map $X_j\ra B_{j-1}$ is an affine $G$-bundle, 
for any finite subgroup $G\subset \mathsf{G}(k)$ - its fibers are given 
by a system of {\em linear} equations on the coordinates of $v_j$. 
 \end{proof}

\begin{prop}
\label{prop:so}
Let $\mathsf G=\mathsf{SO}_n$ and $G\subset \mathsf{G}(k)$ be a finite subgroup.
Then there exist a $G$-ruling $X$, a variety $Y$ with trivial $G$-action and a $G$-equivariant
finite morphism
$$
\pi\,:\, \mathsf G\times Y\ra X.
$$
Moreover, $\deg(\pi) \mid 2^{n-1}$. 
\end{prop}

\begin{proof}
Keep the notations in the proof of Proposition~\ref{prop:g-afff}:
$\mathsf G\subset \mathsf{M}_{n}$. 

Assume that $\char(k)\neq 2$. Every quadratic form can be diagonalized
over $k$. Let $X_n\subset \mathsf{M}_n$ be the subvariety given by:
\begin{equation}
\label{eqn:ort-eq}
(v_i,v_{i'})=\delta_{i',n+i},\,\, \text{ for }\,\, i<i'.
\end{equation}
The system of projections is the same as above: $X_j\subset \oplus_{i=1}^j V^{(i)}$
is the subset of vectors satisfying equations 
\eqref{eqn:ort-eq}, for indices $1\le i<i'\le j$, 
and the condition that $v_1,\ldots, v_{j-1}$ are linearly independent.
The subvariety $B_j\subset X_j$ corresponds to $j$-tuples  
$(v_1,\ldots, v_j)$  
which are linearly independent (as vectors in $V=V^{(i)}$).   
Each map $X_j\ra B_{j-1}$ is a $G$-equivariant vector bundle, 
for any finite subgroup $G\subset \mathsf{G}(k)$.
Each $X_j$ carries the action of a $j$-dimensional torus $\mathbb G_m^j$, over $k$, 
commuting with the action of $G$. The action of $ \mathsf G\times \mathbb G_m^n$ on 
$X_n$ is transitive, and the stabilizer of a general $k$-point has order $2^{n-1}$. 
The claim follows, for $Y:=\mathbb G_m^n$. 

Assume that $\char(k)=2$. In this case, $\mathsf{SO}_{2n+1}\simeq \mathsf{Sp}_{2n}$ and 
we can apply Proposition~\ref{prop:g-afff}. We also have
$\mathsf{SO}_{2n}\subset \mathsf{Sp}_{2n}$, where 
$\mathsf{Sp}_{2n}(k)$ is the set of elements of $\mathsf{GL}_{2n}(k)$ 
which preserve a symplectic bilinear form $\omega$, and  
$\mathsf{SO}_{2n}(k)$ the set of 
those elements which in addition preserve a quadratic form $f$. 
The forms are related by the condition
$$
f(x+y)=f(x)+f(y)+\omega(x,y).
$$
We may identify a general element 
$\gamma\in\mathsf{Sp}_{2n}(k)$ with
a choice of an orthogonal basis $\{v_1, \ldots, v_{2n}\}$. 
Observe that the map 
$$
\begin{array}{cccc}
\mathsf{Sp}_{2n}       & \ra     &  \mathsf{Sp}_{2n}/\mathsf{SO}_{2n} & \sim \mathbb A^{2n}\\
\{v_1, \ldots, v_{2n}\}& \mapsto & (f(v_1),\ldots, f(v_{2n})).         & 
\end{array}
$$
Indeed, $\gamma\in \mathsf{SO}_{2n}(k)$ if and only if $f(\gamma x)=f(x)$, for all $x\in V$. 
We have
\begin{align*}
f(\sum_{i=1}^{2n} a_iv_i) & =\sum_{i=1}^{2n} a_i^2f(v_i) +\sum_{i\neq j}a_ia_j(v_i,v_j) \\
                          & =\sum_{i=1}^{2n} a_i^2f(\gamma v_i) +
                          \sum_{i\neq j}a_ia_j\omega(\gamma v_i,\gamma v_j)\\
                          & = f(\gamma (\sum_{i=1}^{2n} a_iv_i)),
\end{align*}
since $f(\gamma v_i)=f(v_i)$ and $\gamma$ preserves $\omega$.

We claim that the bundle $\mathsf{Sp}_{2n}\ra \mathsf{Sp}_{2n}/\mathsf{SO}_{2n}\sim \mathbb A^{2n}$ 
admits a multisection of degree $2^{2n}$. 
Explicitly, it can be constructed as follows:
fix an orthogonal basis $\{v_1, \ldots, v_{2n}\}$ 
such that $f(v_i)\neq 0$, for $i=1, \ldots, 2n$. 
We have an action of the affine group $\mathsf{B}=\mathbb G_m\rtimes \mathbb G_a\subset \mathsf{SL}_2$
given by 
$$
(x_i,x_{n+i})\mapsto (\lambda x_i, \mu x_i+\lambda^{-1}x_{n+i}), \,\,\, \text{ for }\,\,\, i=1,\ldots, n.
$$
We claim that this gives a generically surjective map:
$$
\mathsf{SO}_{2n}\times \mathsf{B}^n \ra \mathsf{Sp}_{2n}
$$
of degree $2^{2n}$.  The image of $\mathsf{B}^n \cdot \{f(v_1), \ldots, f(v_{2n})\}$ is dense in 
$\mathbb A^{2n}$. Consider the intersection $\mathsf{B}^n \cap \mathsf{SO}_{2n}$:
$$
f(v_i)=\lambda^2f(v_i), \,\, \text{ and }\,\, 
f(v_{i+n})=\mu^2f(v_i)+\lambda\mu\omega(v_i,v_{i+n}) + \lambda^{-2}f(v_{i+n}).
$$
These equations can be solved in $k$, for each $i=1,\ldots, n$, and 
we have a dominant map 
$\mathsf{SO}_{2n}\times \mathsf{B}\ra \mathbb A^2_i$ of degree 4, for each $i$.    
This concludes the proof.
 \end{proof}

\section{Stable cohomology}
\label{sect:coho}

In this section we collect background material on the stable
cohomology of finite groups, developing the theory over
arbitrary algebraically closed fields $k$. 
We will omit $k$ from the notation when the
field is clear from from the context.

\

For every finite group $G$ and a $G$-module $M$ we have
the notion of {\em group cohomology}, as the derived functor 
$M\mapsto M^G$, the $G$-invariants, or, topologically,  
as the cohomology of the classifying space $BG=X/G$, where $X$ is 
a contractible space with a fixed point free action of $G$.  

\

Passing to algebraic geometry, 
let $X$ be an algebraic variety over $k$, 
with an almost free action of a $G$. Let $X^{\circ}\subset X$ 
be the locus where the action is free. 
Let $M$ be a finite $G$-module. It defines a sheaf on 
$\tilde{X}:=X^{\circ}/G$. 
This gives a homomorphism from group cohomology of $G$ to \'etale 
cohomology of $\tilde{X}$:
$$
H^i(G,M)\ra H^i_{et}(\tilde{X}, M).
$$ 
Composing with restriction to the generic point we get a homomorphism
$$
\sigma_i^*\,:\, H^i(G,M)\ra H^i(\G_K,M),
$$ 
where $\G_K$ is the absolute Galois group of 
the function field $K=k(\tilde{X})$. 
There are canonical isomorphisms
$$
H^i(\G_K,M)= \lim_{\stackrel{\longrightarrow}{\tiny D}}H^i(X\setminus D,M),
$$
where the limit is taken over divisors of $D$. We can interpret elements
in the kernel of $\sigma^*$ as classes vanishing on some 
Zariski open subvariety $\tilde{U}\subset \tilde{X}$.

\begin{rem}
Note that for fixed $G$ and $M$, the groups 
$\sigma_i^*(H^i(G,M))=0$, for all $i>\dim(X)$, while
the usual group cohomology need not vanish. 
\end{rem} 

\begin{prop}
\label{prop:finite}
There exist a finite group $\tilde{G}$ and a 
sequence
$$
\G_K\stackrel{\tilde{\sigma}}{\longrightarrow} \tilde{G} \stackrel{\rho}{\longrightarrow} G
$$
of homomorphisms $\tilde{\sigma}$ and $\rho$ such that for all $0\le i\le \dim(X)$ one has
$$
\Ker(\sigma_i^*)\subset \Ker(\rho_i^*),
$$
where $\rho_i^*\,:\, H^i(\tilde{G},M)\ra H^i(G,M)$ is the induced map on 
group cohomology. 
\end{prop}

\begin{proof}
The cohomology classes in $H^i (\G_K, \Z/\ell)$ 
are represented by continuous cocycles 
(in the natural topology on $\G_K$).
Any element is induced from a finite group $H$. If it vanishes
it also vanishes on a finite quotient $\tilde{G}$ of $\G_K$ and
the maps  $\G_K \to \tilde{G} \to G$ are continuous.
Since the initial group $H^i(G, \Z/\ell)$ is finite there exists a
$\tilde{G}$ where all elements from $H^i(G , \Z/\ell)$, which
vanish on $\G_K$ are killed.
 \end{proof}

A special case of the above construction arises as follows:
let $\varrho\,:\, G\ra V$ be a 
faithful representation over an algebraically 
closed field $k$ and let $K=k(V)^G$ 
be the function field of the quotient. 
We have induced maps
$$
s_i^*\,:\, H^i(G,M)\ra H^i(\G_K,M)
$$
and we can define the {\em stable cohomology groups} over $k$: 
$$
H^i_{k,s}(G,M):=H^i(G,M)/\Ker(s_i^*),
$$ 
which we
will often identify with their image in $H^i(\G_K,M)$.

\begin{prop}
\label{prop:div}
The cohomology groups $H^i_{k,s}(G,M)$ 
\begin{itemize}
\item[(1)] do not depend on the representation;
\item[(2)] are functorial in $G$;
\item[(3)] are universal for $G$-actions:
for any $G$-variety $X$ over $k$ the homomorphism 
$H^i(G,M)\ra H^i(\G_{k(X)},M)$ factors through $H^i_{k,s}(G,M)$;
\item[(4)] if $M$ is an $\ell$-torsion module, then 
$$
H^i_{k,s}(G,M)=H^i_{k,s}(\Syl_{\ell}, M)^{\mathfrak N_{\ell}}
$$
where $\Syl_{\ell}=\Syl_{\ell}(G)$ is an $\ell$-Sylow subgroup of $G$
and  $\mathfrak N_{\ell}=\mathfrak N_{\ell}(G)$ its normalizer in $G$. 
\end{itemize}
\end{prop}

\begin{proof}
We apply Lemma~\ref{lemm:openU}.
Choosing an appropriate Zariski open $G$-invariant subvariety $X^{\circ}\subset X$
we can reduce to the affine case, with free $G$-action. Let $V^{\circ}\subset V$
be a Zariski open subset where the action of $G$ is free. Put
$\tilde{X}:=G\ba X^{\circ}$ and $\tilde{V}:=G\ba V^{\circ}$. 
We need to show that a class
$\alpha\in H^i(G,M)$ whose image in $H^i(\G_{k(\tilde{V})}, M)$ is zero
also vanishes in $H^i(\G_{k(\tilde{X})}, M)$. 
Such a class vanishes in $H^i_{et}(\tilde{U},M)$, where $\tilde{U}\subset \tilde{V}$ is 
an affine Zariski open subset. 
The preimage $U$ of $\tilde{U}$ in $V$ is a nonempty $G$-invariant affine Zariski open subset. 
Thus there exist an affine nonempty $G$-invariant Zariski open subset $U_X\subset X^{\circ}$
and a $G$-morphism $\phi_U\,:\, X^{\circ}\ra V$ such that $\phi_U(U_X)\subset U$. 
This descends to a morphism $\tilde{X}\supset \tilde{U}_X\ra \tilde{U}\subset \tilde{V}$.
The image of $\alpha$ under the composition
$$
H^i(G,M)\ra H^i_{et}(\tilde{U},M)\ra H^i_{et}(\tilde{U}_X,M)\ra H^i_{et}(\G_{k(\tilde{X})},M)
$$ 
is zero. This proves (3). Applying this to $X=V'$, for another faithful representation, 
we get (1). 

Property (2) is proved as follows: First, let $H\subset G$ be a subgroup and $V$ a faithful
$G$-representation. Consider the morphism $H\ba V\ra G\ba V$. A class vanishing on a Zariski
open subset of $G\ba V$ also vanishes on a Zariski open subset of $H\ba V$.  Next, let $G\ra H$ be
a surjective homomorphism and $V_G$, resp. $V_H$, a faithful representation of $G$, resp. $H$. 
Then $W_G:=V_H\oplus V_G$ is a faithful representation of $G$ and we have a commutative
diagram
$$
\begin{array}{ccc}
\G_{k(G\ba W_G)} & \ra   & \G_{k(H\ba V_H)} \\ 
\downarrow    &       & \downarrow \\ 
        G     & \ra   &  H,
\end{array} 
$$
giving natural maps on cohomology.

We proceed with the proof of Property (4).
Since $\ell$ and the cardinality of $G/\Syl_{\ell}$ are coprime, 
the map $\Syl_{\ell}\ba V\ra G\ba V$
induces an invertible map on the cohomology of the open subvarieties of 
$G\ba V$. 
The group  ${\mathfrak N}_{\ell}(G)/ \Syl_\ell(G)$ 
has order prime to $\ell$. The action of 
${\mathfrak N}_{\ell}(G)/ \Syl_\ell(G)$ on $M$ decomposes the 
module into a direct sum; so that  
$$
H^i(\Syl_{\ell}(G), M) =H^i(\Syl_{\ell}(G), M)^{{\mathfrak N}_{\ell}(G)} \oplus R,
$$ 
so that the restriction of the trace map is zero on the module $R$. 

We have  
$$
H^i(G,M)\stackrel{\sim}{\longrightarrow} 
H^i(\Syl_{\ell}(G),M)\subset  
H^i(\Syl_{\ell}(G), M)^{{\mathfrak N}(\Syl_{\ell}(G))}.
$$
Consider the image of  $r\in R$ in  $H^*_{k,s}(\Syl_{\ell}(G),M)$.
We get a direct decomposition
$H^i_s(\Syl_{\ell}(G), M)^{\mathfrak N(\Syl_{\ell}(G))} \oplus R_s$, with 
${\rm Tr}(r) = 0$.
Thus $H^i_{k,s}(G,M)$ surjects onto 
$H^i(\Syl_{\ell}(G), M)^{{\mathfrak N}(\Syl_{\ell}(G))}$,
and the map is an isomorphism.

 \end{proof}

\begin{lemm} 
Let $V$ be a representation space for a faithful representation of 
group $G$ over an algebraically closed field $k$.
Assume that $G\backslash V$ is isomorphic to affine space.  
Then any nontrivial element $\alpha\in H_{k,s}(G,\Z/\ell)$ 
has nontrivial restriction to the stable cohomology of 
a centralizer of a quasi-reflection in $G$.
\end{lemm}

\begin{proof} 
If $\alpha\in H^*_{k,s}(k(\A^n),\Z/\ell)$ is nontrivial, then the
residue of $\alpha$ is nontrivial 
on some irreducible divisor $D\subset \A^n=G\backslash V$
(see \cite{c-o}). The preimage of $D$ in $V$ is
a union of irreducible divisors $D_1,\ldots, D_r$.
For each $i$, there exists a nontrivial $\gamma_i\in G$ acting 
trivially on all points of $D_i$. Thus each
$D_i$ is a hyperplane in $V$. Hence $\gamma_i$ is a quasi-reflection.
 \end{proof}

\begin{coro}
Let $\mathcal W$ be a Weyl group. Then 
$$
H^i_{k,s}(\mathcal W,\Z/\ell)=0, \quad \text{  for all }
i>0 \,\, \text{ and all }\,\, \ell\neq 2.
$$
We have
$$
H^i_{k,s}(\mathcal W,\Z/2)\hookrightarrow \oplus_{\tau} 
H^i_{k,s}(\tau, (\Z/2)^{r_{\tau}}),
$$
where $\tau$ runs over the set of 2-elementary abelian subgroups
of $\mathcal W$, modulo conjugation, and $r_{\tau}\in \N$.
\end{coro}

\begin{proof}
The quasi-reflections in the standard faithful representation
of $\mathcal W$ have order 2. 
Their centralizers are products of powers of $\Z/2$ with
smaller Weyl groups.
It suffices to apply induction. 
 \end{proof}

\begin{rem}
It is possible to obtain a more precise vanishing result
following the approach for $\mathcal W=\mathfrak S_n$ 
in \cite{garibaldi}. 
\end{rem}

\section{Comparison with Serre's negligible classes}
\label{sect:comp}

Stable cohomology was defined by the first
author in \cite{b-ams} and \cite{b-mpi}.
J.P. Serre, in his 1990-1991 course at College de France  \cite{serre-col},
defined a related but somewhat different notion.

In his terminology, {\em negligible} elements 
$\alpha\in H^*(G,M)$ are those which are killed under every 
surjective homomorphism $\G_K \to G$, where $K=k(\tilde{X})$ is 
the function field of a quotient $\tilde{X}=G\ba X^{\circ}$.  
Negligible elements form an ideal
in the total ring $H^*(G,M)$. 

We are considering a smaller set of homomorphisms 
$\G_K \to G$, namely from Galois groups of fields of type $K=k(V)^G$, 
and the ideal of negligible classes defined by Serre is smaller. 
The resulting  groups are different 
(for example, for $\Z/2$-coefficients).
The quotient ring obtained by Serre's construction 
surjects onto the ring $H^*_{k,s}(G ,M)$, for any algebraically closed
$k$.

\section{Unramified cohomology} 
\label{sect:unram}

Here we study unramified cohomology, which was introduced in 
\cite{bogomolov} and \cite{c-o} 
(see also \cite{c-bir}).
Let $K=k(X)$ be a function field over an algebraically closed field $k$, 
and $M$ an \'etale sheaf on $X$.
For every divisorial valuation $\nu\in \Val_K$ of $K$ we have
a split exact sequence 
$$
1\ra \I_{\nu}\ra \G_{K_\nu}\ra \G_{\KK_{\nu}}\ra 1
$$
where $K_{\nu}$ is the completion of $K$ with respect to $\nu$, 
$\KK_{\nu}$ the residue field and $\I_{\nu}$ is the inertia group. 
This gives an exact sequence in Galois cohomology
$$
H^i(\G_{K_\nu}, M^{\I_{\nu}})\ra H^i(\G_{K_{\nu}}, M)
\stackrel{\delta_{\nu}}{\longrightarrow} H^{i-1}(\G_{\KK_{\nu}}, {}_{\I_{\nu}}M)
$$
where $M^{\I_{\nu}}$, resp. ${}_{\I_{\nu}}M$, are the sheaves of invariants, resp. coinvariants.  
Unramified cohomology is defined by 
$$
H^i_{k,un}(\G_K,M):=\cap_{\nu\in \Val_K} \Ker(\delta_{\nu})\subset H^i(\G_K,M).
$$

\begin{lemm}
\label{lemm:funct}
Let $\pi\,:\, X\ra Y$ be a surjective morphism of 
algebraic varieties over $k$, and $M$
an \'etale sheaf on $Y$. Then there is a natural homomorphism:
$$
\pi^*\,:\, H^i_{k,un}(Y,M)\ra H^i_{k,un}(X, \pi^*(M)).
$$
Moreover, if $\pi$ is finite, then 
there is a natural homomorphism
$$
\pi_*\,:\, H^i_{k,un}(X,\pi^*(M))\ra H^i_{k,un}(Y, M)
$$
and the composition $\pi_*\circ \pi^*$ 
is multiplication by the degree of $\pi$. 
\end{lemm}

\begin{proof}
We have an embedding $\pi^*\, :\, k(X)\hookrightarrow k(Y)$
of function fields and 
a the corresponding map $\pi_*\,:\, \G_{k(Y)} \ra \G_{k(X)}$
of Galois groups.
A divisorial valuation $\nu$ of $k(Y)$ is either 
trivial on $\pi^*(k(X))$ or defines
a divisorial valuation $\nu^*$ on $k(X)$.
If $\nu$ is trivial on $\pi^*(k(X))$, then $\pi_*(\I_{\nu})$
for the inertia subgroup $\I_{\nu} \subset \G_{k(Y)_{\nu}}$
and hence $\delta_{\nu}$ is zero on $\pi^* H^*(\G_{k(X)},M)$.
If $\nu$ on $\pi^*(k(X))$ coincides with $\nu^*$ then
under the induced map $\pi_*\, :\, \G_{k(Y)_{\nu}} \ra \G_{k(X)_{\nu'}}$
we have $\pi_*(\I_{\nu}) \subset \I_{\nu^*}$.
Thus $\delta_{\nu^*} (\alpha) = 0$ implies that 
$\delta_{\nu} \pi^*(\alpha) = 0$
for all $\alpha\in H^*_{k,un}(\G_{k(X)},M)$.
 \end{proof}

\

Let $G$ be a finite group and $V$ a faithful $G$-representation as above. 
Let $K=k(V)^G$ be the function field of the quotient.
We can consider its stable cohomology groups 
$H^i_{k,s}(G,M)$ as subgroups of $H^i(\G_K, M)$.
Define {\em unramified cohomology groups}
$$
H_{k,un}(G,M):=H^i_{k,s}(G,M)\cap H^i_{k,un}(\G_K,M).
$$

\begin{prop}
\label{prop:unramcoh}
Assume that $\char(k)\nmid |M|$.
Then the unramified cohomology groups $H^i_{k,un}(G,M)$
\begin{itemize}
\item do not depend on the representation $V$;
\item are functorial in $G$.
\end{itemize} 
\end{prop}

\begin{proof}
Let $V,V'$ be two faithful representations of $G$ and consider the diagram:
$$
\begin{array}{ccc}
G\ba (V\times V') & \ra & G\ba V' \\
\downarrow     &     &   \\
G\ba V.            &     &
\end{array}
$$ 
The case of constant coefficients follows from \cite{ct} 
and the observation that both arrows in the above diagram 
are natural vector bundles on the quotients. A small modification of the argument
proves the claim for a general $G$-module $M$. 
 \end{proof}

\section{General vanishing results}
\label{sect:gen-van}

In this section, we work over $k=\bar{\F}_p$. 
Here we collect general arguments proving the triviality of 
stable and unramified cohomology groups.

\begin{thm}
\label{thm:char-p}
Let $G$ be a finite group and 
$M$ a finite $p$-torsion $G$-module.  
Then, $H^i_{k,s}(G,M)=0$, for all $i>0$. 
\end{thm}

\begin{proof}
See \cite[Chapter 2, Proposition 3]{serre-cohgal}.
 \end{proof}

The main reason for introducing the unramified cohomology group is:

\begin{thm}
\label{thm:reason}
Let $V$ be a faithful representation of $G$. 
If $K=k(V)^G$ is a purely transcendental extension of $k$, then, 
for all $i>0$, we have 
$$
H^i_{k,un}(G,\Z/\ell)=0.
$$
\end{thm}

\begin{proof}
Immediate from \cite[Corollary 1.2.1]{c-o}.
 \end{proof}

\begin{thm}[Lang]
\label{thm:lang}
Let $\mathsf G$ be an algebraic group over $k$.
Let $F$ be an automorphism of $\mathsf G(k)$ 
which is a composition of an element in $\Aut(\mathsf G)(k)$
and a Frobenius of $k=\bar{\F}_p$. 
Let $G=G^{F}\subset \mathsf G(k)$ be the finite subgroup  
fixed by $F$. Then $G\ba \mathsf G\simeq \mathsf G$, 
hence is a rational variety.  
\end{thm}

\begin{proof}
Consider the map 
$$
\begin{array}{rcc}
\tau : \mathsf G & \to &  \mathsf G\\
       x &  \to & F(x)^{-1} x.
\end{array}
$$
The action of $\tau$ on the Lie algebra of $\mathsf G$ 
is surjective with kernel a finite subgroup $G=G^{F}$. 
Note that $\tau$ coincides with
the composition 
$$
\tau: \mathsf G\to G\ba \mathsf G \to \mathsf  G.
$$ 
Indeed, if $\tau(x)=\tau(y)$, then
$F(x)^{-1} x = F(y)^{-1} y$, or
$F(xy^{-1})^{-1} xy^{-1}=1$, or $xy^{-1}\in G$.
If follows that $x = gy, g\in G$. The converse is clear. 
Thus $G\ba \mathsf G$ is rational. 
 \end{proof}

\begin{lemm}
\label{lemm:vanish}
Let $\ell$ be a prime and 
$\pi\,:\, X\ra Y$ a separable morphism $k$-varieties 
of finite degree prime to $\ell$. 
Assume that $H^i_{k,un}(X,\Z/\ell)=0$. 
Then $H^i_{k,un}(Y,\Z/\ell)=0$.
\end{lemm}

\begin{proof}
Immediate from Lemma~\ref{lemm:funct}: the 
degree $\deg(\pi)$ is prime to $\ell$ and 
multiplication by $\deg(\pi)$ is invertible on $H^i_{k,un}(Y,\Z/\ell)$.
 \end{proof}

This lemma will be applied to $Y=G\ba V$. The goal will be to construct $X$
with vanishing unramified cohomology.

\begin{coro}
\label{coro:lang}
Let $G=G^F\subset \mathsf G(k)$ be as above.  
Let $X$ be a $G$-linear variety over $k$. Assume that there exist
a variety $Y$ over $k$, with trivial $G$-action, 
and a $G$-equivariant finite
morphism
$$
\pi\,:\,  \mathsf G\times Y\ra X.
$$
Let $S$ be the set of all primes dividing the degree of $\pi$. 
Then 
$$
H^i_{k,un}(G, \Z/\ell)=0, \,\, 
\text{ for all  } i>0\,\,\text{ and }\,\, \ell\notin S\cup \{p\}.
$$
\end{coro}

\begin{proof}
Any affine connected algebraic group over $k=\bar{\F}_p$ is rational. 
By Lang's theorem, the quotient $G\ba \mathsf G$ 
is isomorphic to $\mathsf G$, and thus
rational. For primes $\ell\notin S\cup\{p\}$ 
not dividing the degree of $\pi$, the induced 
map on cohomology is injective. This concludes the proof. 
 \end{proof}

\begin{thm}
\label{thm:mainn}
Let $\mathsf G$ be a Lie group over $k$. 
Let $G=G^F\subset \mathsf G(k)$ be a finite subgroup.
Put
$$
s(G):=\left\{
\begin{array}{ll}
\{p,2\}   & \text{ for } $G$ \text{ of type } C \text{ or } D_n, n\ge 5;\\
\{p,2,3\} & \text{ for } $G$ \text{ of type } D_4, F_4, E_6, E_7;\\
\{p\}     & \text{ otherwise. } 
\end{array}
\right.
$$ 
Then 
$$
H^i_{k,un}(G,\Z/\ell)=0
$$
for all $\ell\notin s(G)$. 
\end{thm}

We have a natural homomorphism
$$
H^i(G,\Z/\ell)\ra H^i(G\backslash \mathsf G,\Z/{\ell}).
$$
By Lang's theorem \ref{thm:lang}, 
$G\backslash \mathsf G\simeq \mathsf G$, as algebraic varieties.
Thus we get a homomorphism 
$$
\rho\,:\, H^i(G,\Z/\ell)\ra  H^i_{et}(\mathsf G, \Z/\ell).
$$
Assume that $\mathsf G$ is semi-simple. 
Then ${\rm Pic}(\mathsf G)\simeq \pi_1(\mathsf G)$ is a finite group and 
$\mathsf G=\tilde{\mathsf G}/{\rm Pic}(G)$, 
where $\tilde{\mathsf G}$ is the universal cover of $\mathsf G$.
We obtain a natural homomorphism
$$
\eta\,:\, H^i_{et}({\rm Pic}(\mathsf G),\Z/\ell)
\ra H^i_{et}(\mathsf G,\Z/\ell).
$$

\begin{thm}
\label{thm:mainn-2}
Let $\mathsf G$ be a semi-simple Lie group over $k$. 
Let $G=G^F\subset \mathsf G(k)$ be a finite subgroup.
Consider the diagram

\

\centerline{
\xymatrix{        
&              H^i_{et}({\rm Pic}(\mathsf G), \Z/\ell) 
  \\H^i(G,\Z/\ell)\ar[r]^{{\sigma^i_*\circ\,\rho\,\,\,\,\,}} & \ar[u]^{\sigma^i_*\circ\,\eta\,\,\,} H^i_{k,s}(\mathsf G, \Z/\ell).
}
}

\

Then the image of $\rho$ is contained in the image of $\eta$. 
\end{thm}

\begin{proof}
Standard computation using restriction of the 
fibration $\mathsf G\ra \mathsf G/\mathsf T$ to 
$\mathsf G^{\circ}=\mathsf T\times \A^N$
and the transgression homomorphism.
Geometrically, the map $\sigma^* \circ\nu$ in the diagram corresponds
to the embedding of the maximal Bruhat cell $\mathsf U^+\mathsf  T \mathsf U^-$ into the
$G\backslash \mathsf G= \mathsf G$, where $\mathsf T$ is the maximal torus. 
We have 
$$
H^*_{k,s}(\mathsf T, \Z/\ell)= H^*(\mathsf T,\Z/\ell),
$$ 
for $\ell\neq p$, and these coincide
with $H^*({\rm Pic}(\mathsf T)\otimes \Z/\ell, \Z/\ell)$.
 \end{proof}

\begin{coro}
\label{coro:aff-act}
Assume that $\mathsf G$ is simply-connected and that 
the natural translation action of $\mathsf G$ on itself is
affine. Then 
$$  
H^i_{k,s}(\mathsf G,\Z/\ell) =0,
$$
for all $i>0$ and $\ell\nmid q$.
\end{coro}

\begin{proof}
For simply-connected $\mathsf G$ the image 
of $H^*(\mathsf G,\Z/\ell)$ in $H^*_{k,s}(\mathsf T, \Z/\ell)$ is trivial.
This follows from the computation of the transgression homomorphism
for $\mathsf G/\mathsf T$, since $\mathsf G$ is a $\mathsf T$-torsor
over $\mathsf G/\mathsf T$, with ${\rm Pic}(\mathsf G/\mathsf T)$ isomorphic to
the character group of the maximal torus $\mathsf T$.
Geometrically, this corresponds to the fact that 
$\mathsf T$ is contractible in $\mathsf G$, provided $\mathsf G$ is simply-connected.
\end{proof}

\begin{coro}
Let $\mathsf G$ be a semi-simple Lie group over $k$ and
$G\subset \mathsf G(k)$. Then the image of $H^i_{k,un}(G)$ in
$H^i_{k,s}(G\backslash \mathsf G(k)/G, \Z_\ell)$ is trivial.
\end{coro}

\begin{proof}
The group $G$ is contained in a subgroup $G^F$ of $\mathsf G(k)$,
for some $F$, and $G=[\mathsf G(k)^F,\mathsf G(k)^F]$ in this case, with
$G^F/G$ a group of order coprime to $p$.
As before $G^F\backslash \mathsf G= \mathsf G$ as an algebraic variety, which 
contains the cell $\mathsf U^- \mathsf T \mathsf U^+$. Thus $G\backslash \mathsf G$ 
contains a finite abelian unramified covering of $\mathsf U^-\mathsf  T \mathsf U^+$ which is also a product
 $\mathsf U^- \tilde{\mathsf T}\mathsf  U^+$, where $\tilde{\mathsf T}$ is a torus.
The cohomology $H^i_{k,s}(G\backslash \mathsf G, \Z/\ell)$
embedds into the cohomology of $\tilde{\mathsf T}$. Hence $H^i_{k,un}(G,\Z/\ell)$
map to zero.
\end{proof}

\section{Reduction to Sylow subgroups}
\label{sect:splitting}

Let $G$ be a finite group. For $H\subset G$ let $\mathfrak N_G(H)$ denote the normalizer 
of $H$. Let $\Syl_{\ell}(G)$ be an
$\ell$-Sylow subgroup of $G$. 
Recall the following classical result (see, e.g., \cite[Section III.5]{adem-milgram}):
$$
H^i(G,\Z/\ell)=H^i(\Syl_{\ell}(G), \Z/\ell)^{\mathfrak N_G(\Syl_{\ell}(G))}.
$$

\begin{thm}
Let $G$ be a finite group. Let $\ell$ be a prime distinct from the characteristic of $k$. 
Then there is an isomorphism 
$$
H^i_{k,s}(G,\Z/\ell)\stackrel{\sim}{\longrightarrow} 
H^i_{k,s}(\Syl_{\ell}(G),\Z/\ell)^{\mathfrak N_G(\Syl_{\ell}(G))}.
$$
Similarly, 
$$
H^i_{k,un}(G,\Z/\ell)\stackrel{\sim}{\longrightarrow} 
H^i_{k,un}(\Syl_{\ell}(G),\Z/\ell)^{\mathfrak N_G(\Syl_{\ell}(G))}.
$$
\end{thm}

\begin{proof}
Let $V$ be a faithful representation of $G$ over $k$. 
Then the map 
$$
\pi\,:\, \Syl_{\ell}(G)\ba V\ra G\ba V
$$ 
is a finite,
separable and surjective map of degree prime to $\ell$. 
Hence $\pi_*\circ \pi^*$ is invertible in cohomology.
This implies the first claim. 
 
The fact that the local ramification indices of $\pi$ are coprime 
to $\ell$ implies the second claim. 
 \end{proof}

\begin{lemm}
Let $G\subset \mathsf G(k)$ be a finite subgroup of a Lie group
$\mathsf G$ over $k$ which has a trivial intersection with the center of $\mathsf G(k)$.
Let $G^{int}\backslash \mathsf G$ be the quotient of $\mathsf G$ by the conjugation 
action of $G$. 
Then the homomorphisms
$$
h^{int}\,:\, H^i_{k,s}(G,\Z/\ell) \to  H^i_{k,s}(G^{int}\backslash \mathsf G ,\Z/\ell)
$$
and
$$
h^{int}_{un} :H^i_{k,un}(G,\Z/\ell)\to H^{i}_{k,un}(G^{int}\backslash  \mathsf G ,\Z/\ell)
$$ 
are injective.
\end{lemm}

\begin{proof} 
The conjugation action of $G$ on $\mathsf G$ 
is almost free, and the maps $h^{int}, h^{int}_{un}$
are well-defined. The identity 
$e\in G(k)$ is a smooth invariant point for the conjugation action.
In particular, the local (conjugation) action of any
$\ell$-Sylow subgroup of $G$ lifts to characteristic 
$0$ with the same invariant point. The action of $\Syl_\ell(G)$, $\ell\neq p$, 
in a formal neighborhood of $e$ is
equivalent to the linear action on the tangent space $T \mathsf G|_e$ at 
$e$ and lifts to a linear representation of $\Syl_\ell(G)$
in characteristic $0$. 
Since the linear action of $G$ on the linear space $T\mathsf G|_e$ over $k$
is equivariant with respect to the locally free action of $\Syl_\ell(G)$
on $T\mathsf G|_e$, any element of $H^i(\Syl_\ell(G), \Z/\ell)$ 
which is trivial
in the quotient of the formal neighborhood of 
$0\in \hat H^i_{k,s}(T\mathsf G|_e/\Syl_\ell(G),\Z/\ell)$ 
will also be trivial in 
$H^i_{k,s}(T\mathsf G|_e/\Syl_\ell(G),\Z/\ell)= H^i_{k,s}(G, \Z/\ell)$.

This proves the injectivity of $h^{int}$. 
The same argument works for $h^{int}_{un}$.
 \end{proof}

\begin{lemm}
\label{lemm:wg}
Let $G, H$ be finite groups. Let $\rho \,:\, H\ra W$ be a 
faithful $k$-representation of $H$. 
Assume that $H\ba W$ is stably rational.
Let $U$ be a faithful representation of $\tilde{G}:=H\wr_{S} G$, 
where $S$ is a finite $G$-set. 
Then $\tilde{G}\ba U$ is stably birationally equivalent to $G\ba V$, 
where $V$ is a faithful representation of $G$.
\end{lemm}

\begin{proof}
Put
$$
\rho_{S}=\oplus_{s\in S} \rho_s \,:\, H_S:=\prod_{s\in S}  H_s \ra \Aut(W_S), \quad
W_S=\oplus_{s\in S} W_s, 
$$
where $H_s=H$, $W_s=W$,  for all $s\in S$, and $\rho_s=\rho$ on the factor $H_s$
and trivial on $H_{s'}$, for $s'\neq s$. 
We construct $U:= V'\oplus W_S$, where $V'$
is a faithful representation of $G$ and extend the action of $G$ to $V_S$ via
the $G$-action on $S$. This gives a representation of 
$\tilde{G}=H\wr_S G$ in $\Aut(U)$.  
The quotient space is a fibration over $G\ba V'$ with fibers $(H\ba W)^{|S|}$.
We can assume that $H\ba W$ is rational. The action of $G$ on  $(H\ba W)^{|S|}$
permutes the coordinates. It follows that $\tilde{G}\ba U$ is birationally 
equivalent to a vector bundle over $G\ba V'$. 
 \end{proof}

\begin{coro}
Let $G=\Syl_{\ell}(\mathfrak S_n)$ and let $V$ be a faithful representation of $G$. 
Then $G\ba V$ is stably rational. 
\end{coro}

\begin{proof}
The $\ell$-Sylow subgroups of $\mathfrak S_n$ are products of wreath products
of groups $\Z/{\ell} \wr \cdots \wr \Z/\ell$ (see \cite[VI.1]{adem-milgram}).
The quotient $H\ba W$ is rational, for a faithful representation 
$W$ of $H=\Z/\ell$. We apply induction to conclude that
the quotient $G\ba V$ is stably rational. 
 \end{proof}

\begin{coro}
Let $G=\Syl_{\ell}(\mathsf{GL}_n(\F_q))$, with $\ell\nmid q$, and let $V$ be
a faithful representation of $G$. Then $G\ba V$ is stably rational.
\end{coro}

\begin{proof}
The structure of $\ell$-Sylow subgroups of $\mathsf{GL}_n(\F_q)$ is known
(see \cite[VII.4]{adem-milgram}: it is also a product of iterated 
wreath products of cyclic $\ell$-groups. 
Thus we can apply Lemma~\ref{lemm:wg}.
 \end{proof}            

\begin{coro}
\label{coro:gln}
Let $k$ be an algebraically closed field of 
characteristic zero. 
Then 
$$
H^i_{k,un}(\mathsf{GL}_n(\F_q), \Z/\ell)=0 \quad \text{ for all } \,\, i>0, \,\,
\text{ and }\,\, \ell\nmid q.
$$
\end{coro}

\begin{rem}
Similar computations can be performed for some other finite groups of Lie type, 
e.g., for $\mathsf{O}^{\pm}_{2m}(\F_q)$ and $\mathsf{Sp}_n(\F_q)$.
\end{rem}

\providecommand{\bysame}{\leavevmode ---\ }
\providecommand{\og}{``}
\providecommand{\fg}{''}
\providecommand{\smfandname}{and}
\providecommand{\smfedsname}{eds.}
\providecommand{\smfedname}{ed.}

\end{document}